\documentclass[12pt]{amsart}
\usepackage{amsmath}
\usepackage{amssymb}
\usepackage{amsthm}

\input xy
\xyoption{matrix}
\xyoption{arrow}

\theoremstyle{plain}
\newtheorem{theorem}{Theorem}

\newcommand{\abs}[1]{\left\vert#1\right\vert}
\newcommand{\br}[1]{\langle#1\rangle}
\newcommand{\C}{\mathcal{C}}
\newcommand{\Cat}{\mathbf{Cat}}
\newcommand{\D}{\mathcal{D}}
\newcommand{\End}{\mathop{\textnormal{End}}}
\newcommand{\Ob}{{\mathop{\textnormal{Ob}}}}
\def\P{\mathcal{P}}
\newcommand{\Set}{\mathbf{Set}}

\begin{document}
\title{Operads for Symmetric Monoidal Categories}

\author{A.\ D.\ Elmendorf}

\address{Department of Mathematics\\
Purdue University Northwest\\
 Hammond, IN 46323}
 
\email{adelmend@pnw.edu}

\date{\today}

\begin{abstract}
This paper gives an explicit description of the categorical operad whose algebras are precisely symmetric monoidal
categories.  This allows us to place the operad in a sequence of four, thus generating a sequence of four successively stricter concepts of symmetric
monoidal category.  A companion paper will use this operadic presentation to describe a vast array of underlying multicategories for
a symmetric monoidal category.  
\end{abstract}

\maketitle

This is the first of two companion papers with the aim of giving a thorough description of the possible underlying multicategories of a
symmetric monoidal category, all of which will be shown to be canonically isomorphic.  The description, however, relies on having a good
model of the categorical operad whose algebras are precisely symmetric monoidal categories, and this appears to be difficult to find
 in the literature, although perhaps well-known to experts in the field.  This paper gives such a good model and includes it in
 the description of the categorical operads underlying four successively stricter concepts of symmetric monoidal category.
 We do so by first constructing a sequence of non-symmetric operads of sets,
\[
V\to Z\to V_0\to T,
\]
whose algebras are as follows:
\begin{enumerate}
\item
Algebras over $V$ are \emph{magmas}: sets with a specified binary operation but with no other relations.
\item
Algebras over $Z$ are \emph{pointed} magmas: magmas with a specified basepoint, but with no relation
to the binary operation.
\item
Algebras over $V_0$ are \emph{unital} magmas: magmas with a strict left and right unit.
\item
Algebras over $T$ are associative monoids.
\end{enumerate}
We then apply two functors in succession: the left adjoint $L$ to the forgetful functor from symmetric operads (of sets) to non-symmetric
operads, and then degreewise the right adjoint $E$ to the set-of-objects functor from (small) categories to sets.  This gives us a sequence
\[
ELV\to ELZ\to EV_0\to ELT
\]
of categorical operads, each of which has a stronger flavor of symmetric monoidal category as its algebras.  Specifically, 
the algebras over each of these categorical operads are given as follows:
\begin{enumerate}
\item
Algebras over $ELV$ are non-unital symmetric monoidal categories.
\item
Algebras over $ELZ$ are symmetric monoidal categories.
\item
Algebras over $ELV_0$ are strictly unital symmetric monoidal categories.
\item
Algebras over $ELT$ are permutative categories.
\end{enumerate}
Item (4) on this list is a result of May (\cite{May2}, theorem 4.9) and Dunn \cite{Dunn}, since $ELT$ is precisely the categorical Barratt-Eccles
operad used in those papers.  Our main interest is in item (2), which is our main theorem.  The other two have similar proofs to
the one we will give.

The author would like to thank the anonymous referee for a thorough reading of the paper and a number of well-received comments that
improved the exposition considerably.

We begin by describing the sequence of non-symmetric operads of sets underlying our sequence of categorical operads.

Our first non-symmetric operad turns out to be free on one generator in degree 2, but it is convenient to have an explicit model for this
operad, which we call $V$, in honor of another model being the vertices
 of Stasheff's associahedra, first introduced in \cite{Stasheff}.
We start with the specification that $V(0)=\emptyset$ and $V(1)=\{1\}$, where 1 is therefore forced to be 
the identity element for the operad.  Higher degrees for $V$ are defined inductively: for $n\ge2$ we define
\[
V(n)=\coprod_{k=1}^{n-1}V(k)\times V(n-k).
\]
The idea here is that elements of $V(n)$ encode the different ways of completely parenthesizing a sequence of length $n\ge2$;
since any such parenthesization breaks at its ``last multiplication'' into two shorter parenthesizations, we obtain these parenthesizations
as ordered pairs of shorter parenthesizations.  Elements of $V(n)$ then become lists of 1's of length $n$ with a complete 
parenthesization of the list.

The composition operation on $V$, which we write as $\Gamma$ for all operads, is defined inductively as well.  Since $1\in V(1)$ must serve as the identity, 
we must have
\[
\Gamma(1;\alpha)=\alpha
\]
for any $\alpha\in V(n)$ for any $n$.  Now any element $\beta$ of $V$ other than 1 must be in higher degree, and therefore can
be expressed uniquely as an ordered pair $(\beta_1,\beta_2)$ of elements of smaller degree, so we can assume inductively
that composition has already been defined for them.  If we say $(\beta_1,\beta_2)=\beta\in V(n)$, $\beta_1\in V(k)$, and $\beta_2\in V(n-k)$, we then
define
\[
\Gamma(\beta;\alpha_1,\dots,\alpha_n)
=(\Gamma(\beta_1;\alpha_1,\dots,\alpha_k),\Gamma(\beta_2;\alpha_{k+1},\dots,\alpha_n)).
\]
Notice in particular that for the specific (and single) element $(1,1)\in V(2)$, we have
\[
\Gamma((1,1);\alpha_1,\alpha_2)=(\alpha_1,\alpha_2);
\]
in other words, $(1,1)$ produces the ``multiplication'' of parenthesized lists (actually just a juxtaposition.)

The requirements for a non-symmetric operad, say from \cite{May}, pp.\ 1-2, but deleting any mention of symmetric group
actions, are now verified by inductive use of the definition.

\begin{theorem}
The non-symmetric operad $V$ is free on the single generator $(1,1)\in V(2)$.
\end{theorem}

\begin{proof}
Let $W$ be any other non-symmetric operad with a selected element $a\in W(2)$.  Then it suffices to show that there is a unique
map $g:V\to W$ of non-symmetric operads for which $g(1,1)=a$.  

Let $e\in W(1)$ be the identity element.  Then since $g$ is to be a map of operads, we must have $g(1)=e$, and we are required
by hypothesis to have $g(1,1)=a$.  Since these are the only elements in $V(1)$ and $V(2)$, and $V(0)=\emptyset$, it remains to show
that there is only one choice for $g(\beta)$ for $\beta\in V(n)$ for $n>2$.  But we can write $\beta$ uniquely as $(\beta_1,\beta_2)$
for elements in smaller degree, and therefore we must have
\begin{gather*}
g(\beta)=g(\beta_1,\beta_2)=g(\Gamma((1,1);\beta_1,\beta_2))
\\=\Gamma(g(1,1);g(\beta_1),g(\beta_2))
=\Gamma(a;g(\beta_1),g(\beta_2)).
\end{gather*}
Since $\beta_1$ and $\beta_2$ lie in smaller degree than $\beta$, we may assume inductively that $g(\beta_1)$ and $g(\beta_2)$
are already defined, and therefore there is only one possible definition for $g$.  The verification that $g$ is actually a map
of operads proceeds of course by induction.
\end{proof}

As mentioned above, algebras over $V$ are apparently called \emph{magmas}. They are sets with a specified binary operation and no other relations.

Our second operad extends $V$ to an operad $Z$ (still non-symmetric) that turns out to be free on two generators: one in
degree 0 and one in degree 2.  
The idea is to restrict the slots in a parenthesization in which a variable can appear, and then
the rest must be filled with a selected ``identity'' element, which at this point has no further properties.  For notation, let
\[
\P_n(j)=\{S\subset\{1,\dots,j\}:\abs{S}=n\}.
\]
Of course, $\P_n(j)$ is empty unless $j\ge n$.  We define the components of $Z$ as
\[
Z(n)=\coprod_{j\ge n}V(j)\times\P_n(j).
\]
The idea is that an element of $Z(n)$ consists of a parenthesization of a list of length at least $n$, with $n$ slots selected out in which
a variable can appear; the rest are to be reserved for the ``identity.''

Since $Z(n)$ has infinitely many elements for $n\ge0$, we can't define operations inductively on $n$.  However,
every element of $Z(n)$ is uniquely of the form $(\beta,S)$ for $\beta\in V(j)$ and $S\in\P_n(j)$, and we \emph{can} use
induction on $j$, which we'll call the \emph{internal} degree of $(\beta, S)$.  
In particular, there are precisely two elements with $j=1$ (the smallest possible value, since $V(0)=\emptyset)$:
$(1,\emptyset)$ and $(1,\{1\})$.  We consider $(1,\emptyset)$ the generator in degree 0, and $(1,\{1\})$ the identity for
the operad $Z$.  The generator in degree 2 is $((1,1),\{1,2\})$.

In order to define the composition in $Z$, we need to decompose elements in terms of elements of smaller degree, just
as we did for $V$.  The difference is that we use the internal degree rather than the degree of the operad element, which
is generally smaller.  The basic observation is that if we have an element $(\beta,S)$ with $\beta\in V(j)$ for $j\ge2$ and
$S\in \P_n(j)$, we have the unique decomposition $\beta=(\beta_1,\beta_2)$ with, say $\beta_1\in V(k)$ and $\beta_2\in V(j-k)$.
We can then decompose $S$ as $S_1\cup S_2'$, where
\[
S_1=S\cap\{1,\dots,k\}\text{ and }S_2'=S\cap\{k+1,\dots,j\}.
\]
Now we just subtract $k$ from each element of $S_2'$ to obtain a subset $S_2\subset\{1,\dots,j-k\}$, and define
\[
(\beta,S)=(\beta_1,S_1)\cdot(\beta_2,S_2).
\]
This decomposition is unique, and allows us to define composition inductively, as follows.  First,
since we now have elements in total degree 0, they compose with an empty list, and in particular we define
\[
\Gamma((1,\emptyset);)=(1,\emptyset)
\]
to give us the fundamental nullary operation.  Since $(1,\{1\})$ is to be the identity, we must have
\[
\Gamma((1,\{1\});(\alpha,T))=(\alpha,T)
\]
for any $(\alpha,T)$ an element of $Z$ in any degree.

All other elements $(\beta,S)$ have higher internal degree, and therefore decompose as $(\beta_1,S_1)\cdot(\beta_2,S_2)$
into elements of lower internal degree.  Let's say that $\abs{S_1}=k$ and $\abs{S_2}=n-k$.  We can now define composition inductively as
\begin{gather*}
\Gamma((\beta,S);(\alpha_1,T_1),\dots,(\alpha_n,T_n))
\\=\Gamma((\beta_1,S_1);(\alpha_1,T_1),\dots,(\alpha_k,T_k))\cdot\Gamma((\beta_2,S_2);(\alpha_{k+1},T_{k+1}),\dots,(\alpha_n,T_n)).
\end{gather*}

As with the operad $V$, the verifications that this definition satisfies the necessary identities follow from repeated induction.

\begin{theorem}
The non-symmetric operad $Z$ is free on one generator $(1,\emptyset)$ in degree 0 and one generator $((1,1),\{1,2\})$ in degree 2.
\end{theorem}

\begin{proof}
Suppose given a non-symmetric operad $W$, together with an element $p\in W(0)$ and an element $m\in W(2)$.  It suffices to show that
there is a unique map of operads $g:Z\to W$ for which $g(1,\emptyset)=p$ and $g((1,1),\{1,2\})=m$.

Let $e\in W(1)$ be the operad identity for $W$.  Then we must have $g(1,\{1\})=e$, and together with the requirement that $g(1,\emptyset)=p$,
this provides us with the initial conditions for induction on internal degree.  Note that in analogy to the situation for $V$, we have
\[
\Gamma(((1,1),\{1,2\});(\beta_1,S_1),(\beta_2,S_2))=(\beta_1,S_1)\cdot(\beta_2,S_2). 
\]
Therefore when we have $(\beta,S)=(\beta_1,S_1)\cdot(\beta_2,S_2)$, where $\beta$ has degree at least 2, we must have
\[
g(\beta,S)=g(\Gamma(((1,1),\{1,2\});(\beta_1,S_1),(\beta_2,S_2)))
=\Gamma(m;g(\beta_1,S_1),g(\beta_2,S_2)),
\]
where this expression has been previously defined by induction.  As before, the verification that this defines a map of operads proceeds
also by induction.
\end{proof}

There is a canonical map of operads $V\to Z$ given by sending the free generator in dimension 2 in $V$ to the free generator in
dimension 2 in $Z$.  In terms of the explicit descriptions, it sends an element $\beta\in V(n)$ to $(\beta,\{1,\dots,n\})\in Z(n)$.

The algebras over $Z$ are \emph{pointed} magmas: sets with a specified binary operation and a selected element, but with no
relations.  The map of operads $V\to Z$ induces the forgetful functor that ignores the selected element.

Our third non-symmetric operad we agree to call $V_0$, since it simply adjoins an element 0 to $V$.  We declare $V_0(0)=\{0\}$, and $V_0(n)=V(n)$ for $n\ge1$.
What remains is to define the composition operation, and in order to do so, we need a binary operation
\[
V_0(j)\times V_0(k)\to V_0(j+k),
\]
but we can't just form ordered pairs in $V$ when one of the coordinates is 0.  We solve this problem by 
defining an operation given by the following formula:
\[
\beta_1\cdot\beta_2=
\begin{cases}
\beta_1&\text{ if }\beta_2=0,\\
\beta_2&\text{ if }\beta_1=0,\\
(\beta_1,\beta_2)&\text{ otherwise.}
\end{cases}
\]
We now can define composition by induction basically as we did with $V$.  Remembering that an element in degree 0
composes with the empty list, we declare that 
\[
\Gamma(0;)=0\text{ and }\Gamma(1;\alpha)=\alpha
\]
to start the induction. Then inductively, if $\beta\in V_0(n)$ for $n\ge2$, we must have $\beta=(\beta_1,\beta_2)$ where
$\beta_1\in V_0(k)$ and $\beta_2\in V_0(n-k)$ for $1\le k\le n-1$, and we define
\[
\Gamma((\beta_1,\beta_2);\alpha_1,\dots,\alpha_n)
=\Gamma(\beta_1;\alpha_1,\dots,\alpha_k)\cdot\Gamma(\beta_2;\alpha_{k+1},\dots,\alpha_n).
\]
As before, the requirements for an operad composition are now verified by induction.

The algebras over $V_0$ are \emph{unital} magmas: sets with a binary operation and a specified
unit (both left and right), but no other relations. 

There is a canonical map of operads $Z\to V_0$ sending the free generator in dimension 2 to $(1,1)\in V_0(2)$, and the free generator
in dimension 0 to $0\in V_0(0)$.   The map of operads induces the forgetful functor that remembers the unit element,
but forgets its unit properties.

Our fourth and final non-symmetric operad is the terminal one $T$: its components $T(n)$, $n\ge0$, all consist of a single point.  Its algebras are strictly
associative monoids.  The comparison map $V_0\to T$ is the terminal map collapsing all elements of $V_0(n)$ to a single point.  
The induced forgetful functor forgets the associativity of the product, but remembers the unit element and its unit properties.

We now have the following sequence of maps of non-symmetric operads of sets:
\[
V\to Z\to V_0\to T.
\]
Our next step is to apply the free symmetric operad functor $L$ to each of these non-symmetric operads, resulting in a sequence
\[
LV\to LZ\to LV_0\to LT
\]
of operads (of sets).  We recall the construction of $L$ from \cite{EM2}, proof of Theorem 4.7, for the convenience of the reader.
Note that in \cite{EM2}, this functor is called $L''$, and applies to general multicategories, not just operads, which are simply multicategories
with a single object.

Suppose $Q$ is a nonsymmetric operad, so has component sets $Q(0), Q(1), Q(2),\dots$.  Then 
$L$ converts $Q$ into a symmetric operad $LQ$ whose component
sets are $LQ(n)=Q(n)\times\Sigma_n$, where $\Sigma_n$ is the $n$'th symmetric group.  The symmetric group action
is just regular (right) multiplication by the $\Sigma_n$ factor.  Composition is now forced by the equivariance requirements
for an operad, where $(r_i,\tau_i)\in LQ(k_i)$:
\begin{align*}
&\Gamma((q,\sigma);(r_1,\tau_1),\dots,(r_n,\tau_n))
\\=(&\Gamma(q;r_{\sigma^{-1}1},r_{\sigma^{-1}2},\dots,r_{\sigma^{-1}n}),\sigma\br{k_1,\dots,k_n}\circ(\tau_1\oplus\cdots\oplus\tau_n)).
\end{align*}

Since $L$ is a left adjoint, the algebras over $Q$ and $LQ$ are the same.
This is because an algebra $X$ over $Q$ consists of a map of non-symmetric operads $Q\to\End(X)$, where $\End(X)$ is the
endomorphism operad of the set $X$.  But since $\End(X)$ supports the structure of a symmetric operad, the map $Q\to\End(X)$
determines and is determined by the adjoint map of symmetric operads $LQ\to\End(X)$, giving $X$ the structure of algebra over $LQ$.

The last step is to apply the right adjoint $E$ to the set-of-objects functor $\Cat\to\Set$ degreewise to each of these operads.
Recall that if $X$ is a set, the category $EX$ has $X$ as its set of objects, and exactly one element in each morphism set $EX(x,y)$,
which is therefore a canonical isomorphism between $x$ and $y$.  Since $E$ is a right adjoint, it preserves products, and therefore operad structures.
We therefore obtain a functor, which we also call $E$, from symmetric operads
of sets to symmetric operads of categories; these last we'll just call categorical operads.
This gives us our sequence of categorical operads
\[
ELV\to ELZ\to ELV_0\to ELT.
\]
These in turn give us a sequence of algebras consisting of categories with successively more restrictive structure, as mentioned
at the beginning of the paper:  
\begin{enumerate}
\item
Algebras over $ELV$ are non-unital symmetric monoidal categories.
\item
Algebras over $ELZ$ are symmetric monoidal categories.
\item
Algebras over $ELV_0$ are strictly unital symmetric monoidal categories.
\item
Algebras over $ELT$ are permutative categories.
\end{enumerate}
Our main interest is in item (2), which we now prove; items (1) and (3) are proven similarly, and (4) was proven in the cited papers 
of May and Dunn.
To avoid having to refer repeatedly to $ELZ$, let's rewrite that categorical operad as $Y$.

\begin{theorem}
Symmetric monoidal structures on a category $\C$ correspond bijectively to algebra structures on $\C$ over $Y$.  The bijection extends
to an isomorphism between the category of $Y$-algebras and the category of symmetric monoidal categories and strict maps.
\end{theorem}

\begin{proof}
Suppose first that $\C$ has the structure of a symmetric monoidal category, so we are given $e\in\Ob\C$, $\mu:\C\times\C\to\C$,
and natural isomorphisms $\alpha:(a\otimes b)\otimes c\cong a\otimes(b\otimes c)$, where we write $a\otimes b$ for $\mu(a,b)$,
as well as $\eta_r:a\otimes e\cong a$, $\eta_l:e\otimes a\cong a$, and $\tau:a\otimes b\cong b\otimes a$, all subject to coherence relations.  
We adopt the convention that when giving an object in $Z$, we refer to the corresponding object in $LZ$
with the identity in the symmetric group attached.  
Note that we obtain as a result an element of the operad of objects of $Y$.

Now we can give $\C$ the structure
of $Y$-algebra as follows.

We send the generating nullary operation $(1,\emptyset)$ in $Z$ to the selected identity $e\in\Ob\C$.  We send the generating binary operation
$((1,1),\{1,2\})$ in $Z$ to the multiplication map $\mu$.  Since the objects of $Y$ are the free operad on these two generators, this completely
describes the action of the objects of the operad.  

The associator $\alpha:(a\otimes b)\otimes c\cong a\otimes(b\otimes c)$ is the image of the unique isomorphism between the two
3-morphisms $(((1,1),1),\{1,2,3\})$ and $((1,(1,1)),\{1,2,3\})$.  
The commutator $\tau:a\otimes b\cong b\otimes a$ is the image of the unique isomorphism
between 
the two 2-morphisms
$(((1,1),\{1,2\}),1_2)$ and $(((1,1),\{1,2\}),t)$, where $1_2$ is the identity element of $\Sigma_2$, and $t$ is the non-identity
element of $\Sigma_2$.  The right unit isomorphism $\eta_r:a\otimes e\cong a$ is the image of the unique isomorphism between the two
1-morphisms $((1,1),\{1\})$ and $(1,\{1\})$, and the left unit isomorphism $\eta_l:e\otimes a\cong a$ is the image of the unique
isomorphism between $((1,1),\{2\})$ and $(1,\{1\})$.  
The extension now to an action of all of $Y$ is the content of the coherence theorem
for symmetric monoidal categories: see \cite{Mac}, XI.1 Theorem 1.

Now suppose given an action of $Y$ on $\C$.  We produce a symmetric monoidal structure on $\C$ in the exact inverse fashion to
the previous construction.  In particular, the nullary operation $(1,\emptyset)$ produces an object $e\in\C$ that provides an identity,
and the binary operation $((1,1),\{1,2\})$ creates an operation $\mu:\C\times\C\to\C$ that will provide us with the product on $\C$.
The right unit isomorphism $\eta_r:a\otimes e\cong a$ is the image of the unique isomorphism between the 1-ary operations $((1,1),\{1\})$ and $(1,\{1\})$,
while the left unit isomorphism $\eta_l:e\otimes a\cong a$ is the image of the unique isomorphism between $((1,1),\{2\})$ and 
$(1,\{1\})$.
The transposition isomorphism is the image of the unique isomorphism between the binary operations $(((1,1),\{1,2\}),1_2)$ and $(((1,1),\{1,2\}),t)$.
And the associator $(a\otimes b)\otimes c\cong a\otimes(b\otimes c)$ is the image of the unique isomorphism between the 3-ary
operations $(((1,1),1),\{1,2,3\})$ and $((1,(1,1)),\{1,2,3\})$.

All the necessary coherence diagrams commute since they are the images of diagrams of $n$-ary operations in $Y$ for a fixed $n$, and
all diagrams commute for any category in the image of $E$.  
In detail, the left-right unit identity diagram
\[
\xymatrix{
e\otimes e\ar[rr]^-{=}\ar[dr]_-{\eta_r}
&&e\otimes e\ar[dl]^-{\eta_l}
\\&e
}
\]
commutes, being a diagram of nullary operations.
The transposition diagram
\[
\xymatrix{
a\otimes e\ar[rr]^-{\tau}\ar[dr]_-{\eta_r}&&e\otimes a\ar[dl]^-{\eta_l}
\\&a
}
\]
is a diagram of 1-ary operations, and therefore commutes.
The unit coherence diagram
\[
\xymatrix{
(a\otimes e)\otimes b\ar[rr]^-{\alpha}\ar[dr]_-{\eta_r\otimes1}
&&a\otimes(e\otimes b)\ar[dl]^-{1\otimes\eta_l}
\\&a\otimes b
}
\]
is a diagram of binary operations, and therefore commutes.
The transposition squared being the identity says
\[
\xymatrix{
a\otimes b\ar[r]^-{\tau}\ar[dr]_{=}&b\otimes a\ar[d]^-{\tau}
\\&a\otimes b
}
\]
commutes, but this is also a diagram of binary operations, so it does commute.  The hexagon
\[
\xymatrix{
(a\otimes b)\otimes c\ar[r]^-{\alpha}\ar[d]_-{\tau\otimes1}
&a\otimes(b\otimes c)\ar[r]^-{\tau}
&(b\otimes c)\otimes a\ar[d]^-{\alpha}
\\(b\otimes a)\otimes c\ar[r]_-{\alpha}
&b\otimes(a\otimes c)\ar[r]_-{1\otimes\tau}
&b\otimes(c\otimes a)
}
\]
is a diagram of 3-ary operations, and therefore commutes, and the pentagon
\[
\xymatrix{
((a\otimes b)\otimes c)\otimes d\ar[rr]^-{\alpha}\ar[d]_-{\alpha\otimes1}
&&(a\otimes b)\otimes(c\otimes d)\ar[d]^-{\alpha}
\\(a\otimes(b\otimes c))\otimes d\ar[r]_-{\alpha}
&a\otimes((b\otimes c)\otimes d)\ar[r]_-{1\otimes\alpha}
&a\otimes(b\otimes(c\otimes d))
}
\]
is a diagram of 4-ary operations, and therefore also commutes.  Therefore $\C$ is provided with the structure of a symmetric
monoidal category.  It is straightforward to see that these two constructions are inverse to each other.

Finally, we observe that this bijection extends to an isomorphism between the category of $Y$-algebras and the category
of symmetric monoidal categories and strict maps.  Suppose given two $Y$-algebras $\C$ and $\D$ and a map $F:\C\to\D$
of $Y$-algebras.  Then $F$ must preserve the unit elements, since they are given as the image of the generating nullary operation
in $Y$.  Further, $F$ must preserve the product, since it given as the image of the generating binary operation in $Y$.  And all
the structure maps for a symmetric monoidal category must also be preserved, again because they are given as the images
of maps between operations in $Y$.  Consequently, $F$ must be a strict map of symmetric monoidal categories.  

Conversely, if $F:\C\to\D$ is a strict map of symmetric monoidal categories, then the entire $Y$-algebra structure must be preserved.
On objects, the generating nullary and binary operations are preserved, since they are given by the unit and product.  And the
morphisms in $Y$ must also be preserved, since the structure morphisms are preserved, and the coherence theorem
for symmetric monoidal categories then shows that all the morphisms are preserved.  This concludes the proof.
\end{proof}

\end{document}